\documentclass[12pt,paper=a4, headsepline,headings=normal,abstract=true]{scrartcl}
\usepackage[english]{babel}
\usepackage[utf8]{inputenc}
\usepackage[T1]{fontenc} 
\usepackage{textcomp}
\usepackage[margin=2cm]{geometry}
\usepackage[reqno]{amsmath}  
\usepackage{amssymb}          
\usepackage{amsthm}          
\usepackage{amstext}          
\usepackage{enumerate} 
\usepackage{enumitem}       
\usepackage{bbm}
\usepackage{pifont}           
\usepackage{framed}
\usepackage{scrpage2}
\usepackage{graphicx}
\usepackage{verbatim}
\usepackage{mathrsfs}

\usepackage{url}
\urldef{\urluni}{\url}{http://www.mathematik.uni-kl.de/fuana/}
\urldef{\emailgrothaus}{\url}{grothaus@mathematik.uni-kl.de}
\urldef{\emailvosshall}{\url}{vosshall@mathematik.uni-kl.de}

\setkomafont{pageheadfoot}{%
\normalfont\normalcolor\itshape\small
}
\setkomafont{pagenumber}{\normalfont\bfseries}

\def\stackunder#1#2{\mathrel{\mathop{#2}\limits_{#1}}}

\makeatletter\@addtoreset{equation}{section}\makeatother

\allowdisplaybreaks

\setenumerate[1]{label=(\roman*)} 
\setenumerate[2]{label=(\alph*)}

\theoremstyle{plain}      \newtheorem{theorem}{Theorem}[section]
                          
                          \newtheorem{proposition}[theorem]{Proposition}
													\newtheorem{condition}[theorem]{Condition}

\theoremstyle{remark}     \newtheorem{remark}[theorem]{Remark}
                          \newtheorem{lemma}[theorem]{Lemma}
													\newtheorem{example}[theorem]{Example}

\theoremstyle{definition} \newtheorem{definition}[theorem]{Definition}

\begin{document} 

\newcommand{\grad}{\nabla}
\newcommand{\D}{\partial}
\newcommand{\E}{\mathcal{E}}
\newcommand{\N}{\mathbb{N}}
\newcommand{\R}{\mathbb{R}_{\scriptscriptstyle{\ge 0}}}
\newcommand{\dom}{\mathcal{D}}
\newcommand{\ess}{\operatorname{ess~inf}}
\newcommand{\cem}{\operatorname{\text{\ding{61}}}}
\newcommand{\supp}{\operatorname{\text{supp}}}
\newcommand{\ca}{\operatorname{\text{cap}}}

\begin{titlepage}
\title{\Large Strong Feller property of\\ sticky reflected distorted Brownian motion}
\author{\normalsize\textsc Martin Grothaus \footnote{University of Kaiserslautern, P.O.Box 3049, 67653
Kaiserslautern, Germany.}~\thanks{\emailgrothaus}~\footnotemark[4] 
 \and \normalsize\textsc Robert Voßhall\footnotemark[1]~\thanks{\emailvosshall}~\thanks{\urluni}}
\date{\small\today}
\end{titlepage}
\maketitle

\pagestyle{headings}

\begin{abstract}
Using Girsanov transformations we construct from sticky reflected Brownian motion on $[0,\infty)$ a conservative diffusion on $E:=[0,\infty)^n$, $n \in \mathbb{N}$, and prove that its transition semigroup possesses the strong Feller property for a specified general class of drift functions. By identifying the Dirichlet form of the constructed process, we characterize it as sticky reflected distorted Brownian motion. In particular, the relations of the underlying analytic Dirichlet form methods to the probabilistic methods of random time changes and Girsanov transformations are presented. Our studies are motivated by its applications to the dynamical wetting model with $\delta$-pinning and repulsion.
\end{abstract}

\section{Introduction}
In \cite{FGV13} the authors constructed via Dirichlet form techniques a reflected distorted Brownian motion in $E:=[0,\infty)^n$, $n\in\mathbb{N}$, with sticky boundary behavior which solves the system of stochastic differential equations
\begin{align} \label{sde!}
dX_{t}^i=\mathbbm{1}_{(0,\infty)}\big(X^i_t\big)\,\Big(\sqrt{2}\,dB^i_t+\partial_i\ln \varrho \big(X_t\big)\,dt\Big)+\frac{1}{\beta}\,\mathbbm{1}_{\{0\}}\big(X^i_t\big)\,dt,\quad i\in I,
\end{align}
or equivalently
\begin{multline} 
dX_{t}^i=\mathbbm{1}_{(0,\infty)}\big(X^i_t\big)\,\Big(\sqrt{2}\,dB^i_t+\partial_i\ln \varrho \big(X_t\big)\,dt\Big)+d\ell_t^{0,i},\\
\text{with}\quad \ell_t^{0,i}:=\frac{1}{\beta}\int_0^t\mathbbm{1}_{\{0\}}\big(X_{s}^i\big)\,ds,\quad i\in I,
\end{multline}
weakly for quasi every starting point with respect to the underlying Dirichlet form. Here $I:=\{1,\ldots,n\}$, $\beta$ is a real and positive constant and $(B^i_t)_{t\ge 0}$ are one dimensional independent standard Brownian motions, $i\in I$. $\varrho$ is a  continuously differentiable density on $E$ such that for all $B \subset I$, $\varrho$ is almost everywhere positive on $E_+(B)$ with respect to the Lebesgue measure and for all $\varnothing\not=B\subset I$, $\sqrt{\varrho\big|_{E_{+}(B)}}$ is in the Sobolev space of weakly differentiable functions on $E_{+}(B)$, square integrable together with its derivative, where $E_{+}(B):=\{ x \in E \big|~x_i >0 \text{ for all } i \in B, ~x_i=0 \text{ for all } i \in I \setminus B \}$. $\varrho$ continuously differentiable on $E$ implies that the drift part $\nabla \ln \varrho$ is continuous on $\{ \varrho > 0 \}$. Moreover, $\ell_t^{0,i}$ is the central local time of the solution to (\ref{sde!}), i.e., it holds almost surely
\begin{align*}
\ell^{0,i}_t=\frac{1}{\beta}\int_0^t\mathbbm{1}_{\{0\}}\big(X_{s}^i\big)\,ds 
&= \lim_{\varepsilon \downarrow 0} \frac{1}{2 \varepsilon} \int_0^t \mathbbm{1}_{(-\varepsilon,\varepsilon)} \big(X_{s}^i\big)\, d\langle X^i \rangle_{s}.
\end{align*}
A solution to the associated martingale problem is even provided under the weaker assumptions that $\varrho$ is almost everywhere positive, integrable on each set $E_+(B)$ with respect to the Lebesgue measure and that the respective Hamza condition is fulfilled.\\
This kind of stochastic differential equation is strongly related to the sticky Brownian motion on the half-line $[0,\infty)$ (which is occasionally also called Brownian motion with delayed reflection or slowly reflecting Brownian motion). In \cite{EP12} the authors study Brownian motion on $[0,\infty)$ with sticky boundary behavior and provide existence and uniqueness of solutions to the SDE system
\begin{align}\label{equEP12}
\left\{
\begin{array}{l}
dX_t=\frac{1}{2}d\ell_t^{0+}\big(X\big)+\mathbbm{1}_{(0,\infty)}\big(X_t\big)\,dB_t\\
\mathbbm{1}_{\{0\}}\,dt=\frac{1}{\mu}\,d\ell_t^{0+}\big(X\big),
\end{array}
\right.
\end{align}
for reflecting Brownian motion $X$ in $[0,\infty)$ sticky at $0$, where $X:=\big(X_t\big)_{t\ge 0}$ starts at $x\in [0,\infty)$, $\mu\in (0,\infty)$ is a given constant, $\ell^{0+}\big(X\big)$ is the right local time of $X$ at $0$ and $B:=\big(B_t\big)_{t\ge 0}$ is the standard Brownian motion. In particular, H.-J.~Engelbert and G.~Peskir show that the system (\ref{equEP12}) has a jointly unique weak solution and moreover, they prove that the system (\ref{equEP12}) has no strong solution, thus verifying Skorokhod's conjecture of the non-existence of a strong solution in this case. For an outline of the historical evolution in the study of sticky Brownian motion we refer to the references given in \cite{EP12} and also to \cite{KPS10}.\\
In view of the results provided in \cite{EP12}, the construction of a weak solution as given in \cite{FGV13} is the only reasonable one. However, the construction via Dirichlet form techniques has the well-known disadvantage that the constructed process solves the underlying stochastic differential equation only for quasi-every starting point with respect to the underlying Dirichlet form. Hence, in the present paper we construct a transition semigroup by Girsanov transformations and investigate its properties in order to strengthen the results of \cite{FGV13}. In this way, we obtain a diffusion with strong Feller transition function which solves (\ref{sde!}) for {\em every} starting point in the state space $E$ and furthermore, we also show an ergodicity theorem for {\em every} starting point in the state space $E$ under the assumptions on the density given in Condition \ref{conditions}. Moreover, we establish connections between the analytic Dirichlet form construction and classical probabilistic methods. Under an additional condition we even conclude uniqueness of weak solutions to (\ref{sde!}) (see Remark \ref{remunique}).\\
In the theory of Dirichlet forms it is a common approach to use results of the regularity theory of elliptic partial differential equations in order to deduce that the associated resolvent and semigroup admit a certain regularity and thereby, it is possible to construct a pointwise solution to the underlying martingale problem or stochastic differential equation for an explicitly known set of starting points under very weak assumptions on the density $\varrho$. For example, this has recently been realized in case of distorted Brownian motion on $\mathbb{R}^d$, $d\in \mathbb{N}$, in \cite{AKR03}, in case of absorbing distorted Brownian motion on $\Omega \subset \mathbb{R}^d$, $d\in \mathbb{N}$, in \cite{BGS13}, in case of reflecting Brownian motion on Lipschitz domains in \cite{FT96} and in case of reflecting distorted Brownian motion on $\Omega \subset \mathbb{R}^d$, $d\in \mathbb{N}$, under some smoothness condition on the boundary $\partial \Omega$ in \cite{FG07} and \cite{BG14}. However, in the present setting which involves not only the Lebesgue measure but also multiple measures on the boundary of the state space $E$ due to the product structure of the problem, the elliptic regularity theory is not yet investigated and from our present point of view the required results are out of reach. For this reason, we choose the probabilistic approach of random time changes and Girsanov transformations in order to obtain a strong Feller transition semigroup which seems to be a new approach in this area.\\
Our results apply to the so-called wetting model with $\delta$-pinning and repulsion (also refered to as the Ginzburg-Landau $\nabla\phi$ interface model with $\delta$-pinning and entropic repulsion) and are of particular interest in view of a scaling limit in dimension $d=1$. More precisely, in a finite volume $\Lambda\subset \mathbb{Z}^d$, $d\in\mathbb{N}$, the scalar field $\boldsymbol{\phi}_t:=\big(\boldsymbol{\phi}_t(x)\big)_{x\in\Lambda}$, $t\ge 0$, is described by the stochastic differential equations
\begin{multline}\label{sde}
d\boldsymbol{\phi}_t(x)=-\mathbbm{1}_{(0,\infty)}\big(\boldsymbol{\phi}_t(x)\big)\sum_{\stackunder{\scriptscriptstyle{|x-y|_{\infty}=1}}{\scriptscriptstyle{y\in\Lambda}}}V'\big(\boldsymbol{\phi}_t(x)-\boldsymbol{\phi}_t(y)\big)\,dt\\
+\mathbbm{1}_{(0,\infty)}\big(\boldsymbol{\phi}_t(x)\big)\sqrt{2}dB_t(x)+d\ell_{t}^{\scriptscriptstyle{0}}(x),\quad x\in\Lambda,
\end{multline}
subject to the conditions:
\begin{align*}
&\boldsymbol{\phi}_t(x)\ge 0,\quad \ell_{t}^{\scriptscriptstyle{0}}(x)\mbox{ is non-decreasing with respect to }t,\quad \ell^{\scriptscriptstyle{0}}_{0}(x)=0,\\
&\int_0^\infty \boldsymbol{\phi}_t(x)\,d\ell_{t}^{\scriptscriptstyle{0}}(x)=0,\\
&\beta \ell_{t}^{\scriptscriptstyle{0}}(x)=\int_0^t\mathbbm{1}_{\{0\}}\big(\boldsymbol{\phi}_s(x)\big)\,ds\quad\mbox{for fixed }\beta> 0,\nonumber\\
\end{align*}
where $\ell_{t}^{\scriptscriptstyle{0}}(x)$ denotes the \emph{local time} of $\boldsymbol{\phi}_t(x)\mbox{ at }0$. Here $|\cdot|_{\infty}$ denotes the sup-norm on $\mathbb{R}^d$, $V\in C^2(\mathbb{R})$ is a symmetric, strictly convex potential and $\big\{(B_t(x))_{t\ge 0}\,|\,x\in\Lambda\big\}$ are independent standard Brownian motions. In dimension $d=2$ this model describes the wetting of a solid surface by a fluid. More details on interface models are e.g. presented in~\cite{Ga02}, \cite{Fu05}.\\
In \cite[Sect.~15.2]{Fu05} J.D.~Deuschel and T.~Funaki investigated (\ref{sde}) and gave reference to classical solution techniques as developed e.g. in \cite{WaIk89}. The methods provided therein require more restrictive assumptions on the drift part than in our situation (e.g. the drift is assumed to be bounded and Lipschitz continuous), moreover, do not apply directly (the geometry and the behavior on the boundary differ). First steps in the direction of applying \cite{WaIk89} are discussed in \cite{Fu05} by J.-D. Deuschel and T. Funaki.\\
As far as we know the only reference that applies to the system of stochastic differential equations (\ref{sde}) is \cite{Gra88}. By means of a suitable choice of the coefficients the system of equations given by \cite[(II.1)]{Gra88} coincides with (\ref{sde}), but amongst others the drift part is also assumed to be Lipschitz continuous and bounded. For this reason, it is not possible to apply the results of \cite{Gra88} to the setting investigated by J.-D.~Deuschel and T.~Funaki, since the potential $V$ naturally causes an unbounded drift (see also Example \ref{example}). Moreover, the results of \cite{Gra88} do not include strong Feller properties of the associated semigroup. \\

Our paper is organized as follows. In Section \ref{secmain} we state the required conditions on the density as well as our main results. In Section \ref{DFtrans} we recall some facts about sticky Brownian motion and present the connections of the Dirichlet form constructed in \cite{FGV13} to classical methods from probability theory. In particular, we establish relations to random time changes and Girsanov transformations. In Section \ref{Feller} a Feller transition semigroup is constructed under the conditions given in Section \ref{secmain}. This semigroup is used to construct a pointwise solution to (\ref{sde!}) and the corresponding Dirichlet form is identified.  Moreover, in Section \ref{appl} the setting is applied to the dynamical wetting model.

\section{Main results}\label{secmain}

In the following we denote by $dx_i$ the one dimensional Lebesgue measure and by $\delta_0^i$ the Dirac measure in $0$, where $i=1,\dots,n$ gives reference to the component of $x=(x_1,\dots, x_n) \in E=[0,\infty)^n$. Define the product measure $dm_{n,\beta}:= \prod_{i=1}^n (dx_i + \beta \delta_0^i)$ on $(E,\mathcal{B}(E))$. For an $m_{n,\beta}$-a.e. defined function $\varphi$ on $E$ and $\emptyset \neq B \subset I$ we wirte $\varphi_{|E_+(B)} \in H^{1,2}(E_+(B))$ if the restriction of $\varphi$ to $E_+(B):=\{ x \in E \big|~x_i >0 \text{ for all } i \in B, ~x_i=0 \text{ for all } i \in I \setminus B \}$ can be identified with an element in the Sobolev space $H^{1,2}((0,\infty)^{|B|})$. Furthermore, we denote by $d_{\text{euc}}$ the Euclidean metric.\\

We state the following proposition in order to be able to give afterwards suitable conditions on the density $\varrho$:

\begin{proposition} \label{propindep}
There exists a diffusion process $\mathbb{M}^{n,\beta}=(\Omega,\mathcal{F},(\mathcal{F}_t)_{t \geq 0}, (X_t)_{t \geq 0}, (\mathbb{P}^{n,\beta}_x)_{x \in E})$ (called $n$ independent sticky Brownian motions on $[0,\infty)$) solving the SDE
\begin{align*} dX^i_t= \mathbbm{1}_{(0,\infty)}(X^i_t) \sqrt{2} dB^i_t + \frac{1}{\beta}~ \mathbbm{1}_{\{0\}}(X^i_t)dt, \quad i=1,\dots,n,
\end{align*}
for every starting point $x \in E$, where $(B_t)_{t \geq 0}$ is an $n$-dimensional standard Brownian motion, and the transition semigroup $(p_t^{n,\beta})_{t >0}$ of $\mathbb{M}^{n,\beta}$ has the doubly Feller property, i.e., it is a Feller transition semigroup which admits additionally the strong Feller property (see Definition \ref{defDF}). Moreover,
the Dirichlet form associated to $n$ independent sticky Brownian motions on $[0,\infty)$ is given by the conservative, strongly local, strongly regular symmetric Dirichlet form $(\mathcal{E}^{n,\beta},D(\mathcal{E}^{n,\beta}))$, i.e., the closure on $L^2(E;m_{n,\beta})$ of the bilinear form
\[ \mathcal{E}^{n,\beta}(f,g)= \int_E \sum_{i=1}^n \mathbbm{1}_{\{ x_i \neq 0 \}} ~ \partial_i f ~\partial_i g ~ dm_{n,\beta} \ \ \ \text{ for } f,g \in C_c^1(E). \]
\end{proposition} 

\begin{condition} \label{conditions}
$\varrho=\phi^2$ is strictly positive and bounded such that $\varrho \in C^1(E)$ and
\begin{align} \phi_{|E_+(B)}= \sqrt{\varrho_{|E_+(B)}} \in H^{1,2}(E_+(B)) \quad \text{for every } \emptyset \neq B \subset I.
\end{align}
Moreover, for every $t >0$ and every compact set $D \subset E$ it holds 
\begin{align} \lim_{k \rightarrow \infty} \sup_{x \in D} ~\mathbb{E}^{n,\beta}_x(\mathbbm{1}_{\{ \tau_k \leq t \}}~Z_t) =0, \label{condZ}
\end{align}
where $(Z_t)_{t \geq 0}$ is given by
\[ Z_t=\exp\big(\sqrt{2} \sum_{i=1}^n \int_0^t \partial_i \ln \phi(X_s) \mathbbm{1}_{(0,\infty)}(X_s^i) dB_s^i - \sum_{i=1}^n \int_0^t (\partial_i \ln \phi(X_s))^2 \mathbbm{1}_{(0,\infty)}(X_s^i) ds\big) \]
and $\tau_k:= \inf \{ t>0|~ X_t \notin [0,k)^n \}$ with $(X_t)_{t \geq 0}$ and $(B_t)_{t \geq 0}$ as stated in Proposition \ref{propindep}.
\end{condition}

Under the above assumptions on $\varrho$ it holds:

\begin{theorem} \label{thmmain1}
There exits a conservative diffusion process 
\[ \mathbb{M}^{n,\beta,\varrho}=(\Omega,\mathcal{F},(\mathcal{F}_t)_{t \geq 0}, (X_t)_{t \geq 0}, (\mathbb{P}^{n,\beta,\varrho}_x)_{x \in E}) \] 
on $E$ with strong Feller transition function $(p^{n,\beta,\varrho}_t)_{t \geq 0}$, i.e., $p^{n,\beta,\varrho}_t(\mathcal{B}_b(E)) \subset C_b(E)$, such that the associated Dirichlet form is given by the closure of the symmetric bilinear form $(\mathcal{E}^{n,\beta,\varrho},\mathcal{D})$ on $L^2(E;\varrho m_{n,\beta})$, where
\begin{align*} \mathcal{E}^{n,\beta,\varrho}(f,g)&:= \sum_{\emptyset \neq B \subset \{1,\dots,n\}} \mathcal{E}_B(f,g) \\
&= \int_E \sum_{i=1}^n \mathbbm{1}_{\{ x_i \neq 0 \}} ~ \partial_i f ~\partial_i g ~ \varrho dm_{n,\beta} \ \ \ \text{ for } f,g \in \mathcal{D}:=C_c^1(E) 
\end{align*}
with 
\[ \mathcal{E}_B(f,g):= \int_{E} \sum_{i \in B} \partial_i f ~ \partial_i g~\varrho d\lambda_B^{n,\beta}, \]
where $d\lambda_B^{n,\beta}:=\beta^{n-|B|} \prod_{j \in B} dx_j \prod_{j \in B^c} \delta_0^j$. In particular, $(p^{n,\beta,\varrho}_t)_{t \geq 0}$ fulfills the absolute continuity condition \cite[(4.2.9)]{FOT94}, i.e., the transition probabilities $p^{n,\beta,\varrho}_t(x,\cdot)$, $x \in E$, $t>0$, given by $p_t(x,A):=\mathbb{P}^{n,\beta,\varrho}_x(X_t \in A)$, $A \in \mathcal{B}(E)$, are absolutely continuous with respect to $\varrho m_{n,\beta}$.
\end{theorem}

\begin{theorem} \label{thmmain2}
Let $\mathbb{M}^{n,\beta,\varrho}$ be the diffusion process of Theorem \ref{thmmain1}. It holds for each $i=1,\dots,n$
\begin{align} \label{main} X_t^i=X_0^i + \sqrt{2} \int_0^t \mathbbm{1}_{(0,\infty)}(X_s^i) dB_s^i + \int_0^t \mathbbm{1}_{(0,\infty)}(X_s^i)~ \partial_i \ln \varrho(X_s) ds + \frac{1}{\beta} \int_0^t \mathbbm{1}_{\{0\}}(X_s^i) ds
\end{align}
$\mathbb{P}^{n,\beta,\varrho}_x$-a.s. for every $x \in E$, where $(B_t^i)_{t \geq 0}$, $i=1,\dots,n$, are independent standard Brownian motions. Moreover, it holds
\begin{align} \label{eqnergo} \lim_{t \rightarrow \infty} \frac{1}{t} \int_0^t F(X_s) ds= \frac{\int_E F \varrho dm_{n,\beta}}{\int_E \varrho dm_{n,\beta}} \end{align}
$\mathbb{P}^{n,\beta,\varrho}_x$-a.s. for every $x \in E$ and $F \in L^1(E; \varrho m_{n,\beta})$.
\end{theorem}

\begin{remark}
Let $\Gamma \subset \partial E$ such that $\int_{\Gamma} \varrho dm_{n,\beta} >0$. Then it follows by (\ref{eqnergo}) that 
\[ \lim_{t \rightarrow \infty} \frac{1}{t} \int_0^t \mathbbm{1}_{\Gamma}(X_s) ds= \frac{\int_{\Gamma}  \varrho dm_{n,\beta}}{\int_E \varrho dm_{n,\beta}} >0 \]
$\mathbb{P}^{n,\beta,\varrho}_x$-a.s. for every $x \in E$. This is for example the case for $\Gamma=\{ x_i=0\}$, $i=1,\dots,n$, and confirms the sticky behavior of the process on the boundary.
\end{remark}

\begin{remark}
\begin{enumerate}
\item All proofs of the following results are still valid if $\varrho$ is not necessarily strictly positive, but Condition \ref{conditions} additionally requires $\text{cap}_{\mathcal{E}^{n,\beta,\varrho}}(\{ \varrho=0\})=0$,
\begin{align} \nabla \ln \phi= \frac{\nabla \phi}{\phi} \in L^{\infty}_{\text{loc}}(E \backslash \{ \varrho=0\};m_{n,\beta}),
\end{align}
$D$ is an arbitrary compact subset of $E \backslash \{ \varrho =0\}$ and $\tau_k$ is defined by
\[ \tau_k:= \inf \{ t>0|~ X_t \notin [0,k)^n \backslash B_{\frac{1}{k}}(\{ \varrho=0\}) \},\]
where $B_{\frac{1}{k}}(\{ \varrho=0\}):=\{ x \in E|~\inf_{y \in \{\varrho=0\}}~ d_{\text{euc}}(x,y) \leq \frac{1}{k} \}$. In this case, a strong Feller process on the state space $E_1:= E \backslash \{ \varrho=0\}$ can be constructed and the corresponding Dirichlet form is defined analogously but on the space $L^2(E_1;\varrho m_{n,\beta})$. The additional condition guarantees that the constructed process (using a Girsanov transformation by $(Z_t)_{t \geq 0}$) never hits the set $\{ \varrho=0\}$. However, in this case it seems hard to verify the conditions for specific densities.
\item Note that Condition \ref{conditions} implies $\phi \in D(\mathcal{E}^{n,\beta})_{\text{loc}} \cap C_b(E)$ and $\ln \phi \in D(\mathcal{E}^{n,\beta})_{\text{loc}}$.
\item Condition \ref{conditions} implies 
\[ \nabla \ln \phi=\frac{\nabla \phi}{\phi} \in L^{\infty}_{\text{loc}}(E;m_{n,\beta}),\] since $\phi$ is assumed to be strictly positive and continuously differentiable.
\item (\ref{condZ}) holds for example if $\sup_{x \in D} \mathbb{E}^{n,\beta}_x(Z_t^p) < \infty$ for some $p>1$ (see Remark \ref{remZ}).
\end{enumerate}
\end{remark}

\begin{remark} Let $\varrho:E \rightarrow (0,\infty)$, $\varrho=\exp(-2H)$, be defined by a potential $V$ with nearest neighbor pair interaction, i.e., $H$ is given by
\begin{align} \label{hamilt} H(x_1,\cdots,x_n)= \frac{1}{4} \sum_{\stackunder{|i-j|=1}{i,j \in \{0,\dots,n+1\}}} V(x_i-x_j),
\end{align}
where $x_0:=x_{n+1}:=0$ and $V:\mathbb{R} \rightarrow [-b,\infty)$, $b \in [0, \infty)$, fulfills the conditions of \cite[(2.2)]{Fu05}:
\begin{enumerate}
\item[(i)] $V \in C^2(\mathbb{R})$,
\item[(ii)] $V$ is symmetric, i.e., $V(r)=V(-r)$ for all $r \in \mathbb{R}$,
\item[(iii)] $V$ is strictly convex, i.e., $c_{-} \leq V^{\prime \prime}(r) \leq c_{+}$ for all $r \in \mathbb{R}$ and some constants $c_{-},c_{+} >0$.
\end{enumerate}
Denote by $\phi:=\sqrt{\varrho}=\exp(-H)$ the square root of $\varrho$.\\

Define $\mathbb{V}^{\prime}(i,x)$ for $i=1,\dots,n$ and $x \in E$ by
\[ \mathbb{V}^{\prime}(i,x):= \sum_{\stackunder{|i-j|=1}{j \in \{0,\dots,n+1\}}} V^{\prime}(x_i-x_j). \]

In this case, Condition \ref{conditions} is fulfilled and the stated results hold accordingly with the drift function given by $\partial_i \ln \varrho=- \mathbb{V}^{\prime}(i,\cdot)$, $i=1,\dots,n$, i.e., $\mathbbm{M}^{n,\beta,\varrho}$ solves (\ref{sde}) for $d=1$. Similarly, we also obtain a solution for general $d \in \mathbb{N}$, since the required conditions on $H$ result from the properties of $V$. $d$ only affects the number of nearest neighbors.
\end{remark}

\section{Sticky Brownian motion and Dirichlet form transformations} \label{DFtrans}

\subsection{Sticky Brownian motion on the halfline} \label{secsticky}

Define the Dirichlet form $(\hat{\mathcal{E}},D(\hat{\mathcal{E}}))$ as the closure of
\[ \hat{\mathcal{E}}(f,g):= \int_{[0,\infty)} f^{\prime}(x) g^{\prime}(x)dx, \ \ f,g \in C_c^1([0,\infty)), \]
on $L^2([0,\infty);dx)$. It is well-known that reflecting Brownian motion is associated to $(\hat{\mathcal{E}},D(\hat{\mathcal{E}}))$ and $D(\hat{\mathcal{E}})=H^{1,2}((0,\infty))$ is the Sobolev space of order one.\\
Let $(\tilde{B}_t)_{t \geq 0}$ be a standard Brownian motion defined on a probability space $(\Omega, \mathcal{F}, \mathbb{P})$. Then $\hat{X}_t:=|x + \sqrt{2} \tilde{B}_t|$, $t \geq 0$, yields reflecting Brownian motion on $[0,\infty)$ starting at $x \in [0,\infty)$ and by Tanaka's formula
\begin{align} \label{refltanaka} \hat{X}_t=x + \sqrt{2} \hat{B}_t + L_t^{0+}, \ \ t \geq 0, 
\end{align}
where $\hat{B}_t:=\int_0^t \text{sgn}(x+ \sqrt{2}\tilde{B}_s) d\tilde{B}_s$, $t \geq 0$, is a standard Brownian motion and $(L_t^{0+})_{t \geq 0}$ is the right local time in $0$, i.e., 
\[ L_t^{0+}= \lim_{\varepsilon \rightarrow 0} \frac{1}{\varepsilon}\int_0^t \mathbbm{1}_{[0,\varepsilon)}(\hat{X}_s) ds \]
almost surely. Here, we differ from classical notation by the factor $\sqrt{2}$ (see also Remark \ref{remsqrt2}). The Dirichlet form associated to $(\hat{X}_t)_{t\geq 0}$ is $(\hat{\mathcal{E}},D(\hat{\mathcal{E}}))$ and $(L_t^{0+})_{t \geq 0}$ is an additive functional which is in Revuz correspondance with the Dirac measure $\delta_0$ in $0$. Consider the additive functional $A_t:= t + \beta L_t^{0+}$, $t \geq 0$, for some real constant $\beta >0$. Note that $A_0=0$ and $A_t \rightarrow \infty$ a.s. as $t \rightarrow \infty$. Then sticky Brownian motion on $[0, \infty)$ is usually constructed by a random time change using the right inverse $(\tau(t))_{t >0}$ of $(A_t)_{t >0}$. More precisely,
$X_t:= \hat{X}_{\tau(t)}$ (starting in $x$) solves the stochastic differential equation
\begin{align} \label{1dstickysde} dX_t= \mathbbm{1}_{(0,\infty)}(X_t) \sqrt{2} dB_t + \frac{1}{\beta}~ \mathbbm{1}_{\{0\}}(X_t)dt,
\end{align}
where $(B_t)_{t \geq 0}$ is a standard Brownian motion.
For details on Feller's Brownian motions and in particular, sticky Brownian motion and its transition semigroup, see e.g. \cite{EP12}, \cite{KPS10} or \cite{Kni81}.\\
In \cite[Chapter 6]{FOT94} and \cite[Chapter 5]{ChFu11} is presented how a random time change by an additive functional affects the underlying Dirichlet form. Let $m_{\beta}$ denote the Revuz measure corresponding to $(A_t)_{t \geq 0}$. Clearly, $dm_{\beta}=dx + \beta \delta_0$. In particular, $m_{\beta}$ has full quasi support $[0,\infty)$. Indeed, $m_{\beta}$ is a smooth measure and every set of measure zero with respect to $m_{\beta}$ is in particular of measure zero with respect to $dx$. As a consequence, every quasi open set of measure zero with respect to $m_{\beta}$ is of zero capacity in view of \cite[Lemma 2.1.7 (ii)]{FOT94}. Thus, the Dirichlet form $(\mathcal{E}^{\beta},D(\mathcal{E}^{\beta}))$ on $L^2([0,\infty);m_{\beta})$ associated to $(X_t)_{t \geq 0}$ has the representation
\[  \mathcal{E}^{\beta}(f,g)=\hat{\mathcal{E}}(f,g) \ \ f,g \in D(\mathcal{E}^{\beta})=D(\hat{\mathcal{E}}) \cap L^2([0,\infty);m_{\beta}). \]
In particular, $D(\mathcal{E}^{\beta})=H^{1,2}((0,\infty)) \cap L^2([0,\infty);m_{\beta}) = H^{1,2}((0,\infty))$ by Sobolev embedding. Moreover, $C_c^1([0,\infty))$ is dense in $D(\mathcal{E}^{\beta})$ by \cite[Theorem 5.2.8(i)]{ChFu11} and thus, it is a special standard core of $(\mathcal{E}^{\beta},D(\mathcal{E}^{\beta}))$. Hence, the closure of 
\begin{align} \label{1dform}
\mathcal{E}^{\beta}(f,g)=  \int_{[0,\infty)} f^{\prime}(x) g^{\prime}(x)~dx= \int_{[0,\infty)} \mathbbm{1}_{(0,\infty)}(x)~f^{\prime}(x) g^{\prime}(x)~dm_{\beta}, \ \ f,g \in C_c^1([0,\infty)),
\end{align}
on $L^2([0,\infty);m_{\beta})$ is the Dirichlet form associated to $(X_t)_{t \geq 0}$.

\begin{remark} \label{remsqrt2}
Note that our notion for the solution to the equations (\ref{refltanaka}) and (\ref{1dstickysde}) as reflecting Brownian motion and sticky reflecting Brownian motion  on $[0,\infty)$ respectively differs by the factor $\sqrt{2}$ from classical literature in view of the underlying SDE (\ref{sde!}). If $(Y^{\gamma}_t)_{t \geq 0}$ solves
\[ dY^{\gamma}_t = \mathbbm{1}_{(0,\infty)}(Y^{\gamma}_t) dB_t + \frac{1}{\gamma} \mathbbm{1}_{\{0\}}(Y_t^{\gamma}) dt ~\text{ for } \gamma >0, \]
we obtain the solution to (\ref{1dstickysde}) by setting $X_t:=\sqrt{2}~ Y^{\sqrt{2}\beta}_t$. This identity is useful in order to derive the resolvent density and transition density for the solution to (\ref{1dstickysde}).
\end{remark}
  
Let $F$ be a locally compact separable metric space and denote by $C_0(F):=\{ f \in C(F)|~ \forall \varepsilon >0~ \exists K \subset F \text{ compact }: |f(x)| < \varepsilon ~ \forall x \in F \backslash K \}$ the space of continuous functions on $F$ vanishing at infinity. We can specify the resolvent and transition semigroup of sticky Brownian motion on $[0,\infty)$. \cite[Corollary 3.10, Corollary 3.11]{KPS10} state the following (see also \cite[Section 6.1]{Kni81}):

\begin{theorem} \label{thmdensity}
The transition function $(p_t^{\beta})_{t >0}$ of sticky Brownian motion on $[0,\infty)$ yields a Feller semigroup on $C_0([0,\infty))$, i.e., $p^{\beta}_t(C_0([0,\infty))) \subset C_0([0,\infty))$ and $\lim_{t \downarrow 0} \Vert p^{\beta}_t f -f \Vert_{\infty} =0$ for each $f \in C_0([0,\infty))$. For $\lambda >0$, $x,y \in [0,\infty)$, the resolvent kernel $r_{\lambda}^{\beta}(x,dy)$ of the Brownian motion with sticky origin (i.e., the solution to (\ref{1dstickysde})) is given by
\begin{align}
r_{\lambda}^{\beta}(x,dy)= \frac{r_{\lambda}^D(x,\frac{y}{\sqrt{2}})}{\sqrt{2}} dy + \frac{1}{2 (\sqrt{\lambda}+ \beta \lambda )} \big(2 e^{-\sqrt{2 \lambda}(x+\frac{y}{\sqrt{2}})}dy + \sqrt{2} \beta ~e^{-\sqrt{2 \lambda} x} \delta_0(dy)\big),
\end{align}
where $r_{\lambda}^D(x,y)= \frac{1}{\sqrt{2\lambda}} (e^{-\sqrt{2\lambda}|x-y|}-e^{-\sqrt{2\lambda}(x+y)})$ is the resolvent density of Brownian motion with Dirichlet boundary conditions.

Furthermore, by the inverse Laplace transform it follows that, for $t >0$, the transition kernel $p^{\beta}_t(x,dy)$ of the Brownian motion with sticky origin is given by 
\begin{align} \label{semigroup}
p^{\beta}_t(x,dy)=\frac{p^D_t(x,\frac{y}{\sqrt{2}})}{\sqrt{2}} dy + \sqrt{2} g_{0,\sqrt{2} \beta}(t,x+\frac{y}{\sqrt{2}})dy + \beta~ g_{0,\sqrt{2} \beta}(t,x)~ \delta_0(dy), 
\end{align}
where $p^D_t(x,y)=p(t,x,y)-p(t,x,-y)$ is the transition density for Brownian motion with Dirichlet boundary conditions, $p(t,x,y)= \frac{1}{\sqrt{2\pi t}} e^{-\frac{(x-y)^2}{2t}}$ and 
\[ g_{0,\gamma}(t,x)= \frac{1}{\gamma} \exp(\frac{2x}{\gamma}+\frac{2t}{{\gamma}^2})~ \textnormal{erfc}(\frac{x}{\sqrt{2t}} + \frac{\sqrt{2 t}}{\gamma}), \ \ \text{for } \gamma >0,~ t >0,~ x \geq 0,\]
with the complementary errorfunction $\textnormal{erfc}(x)=\frac{2}{\sqrt{\pi}} \int_x^{\infty} e^{-z^2} dz$, $x \in \mathbb{R}$.
\end{theorem}

\begin{remark}
Note that (\ref{semigroup}) implies that $p^{\beta}_t(x,\cdot)$ is absolutely continuous with respect to the measure $dm_{\beta}=dx + \beta \delta_0$ for each $x \in [0,\infty)$, $t>0$. Therefore, the so-called {\em absolute continuity condition} \cite[(4.2.9)]{FOT94} is fulfilled. In the following we see that the transition semigroup possesses even stronger properties.
\end{remark}

Thus, with $p^{\beta}_t(x,dy)$ as above and $p^{\beta}_t$, $t >0$, the transition semigroup of sticky Brownian motion, it holds
\[ \mathbb{E}_x(f(X_t))=p^{\beta}_tf(x)= \int_{[0,\infty)} f(y) ~p^{\beta}_t(x,dy) \]
for each $x \in [0,\infty)$ and $f \in C_0([0,\infty))$. Furthermore, the resolvent $r_{\lambda}^{\beta}$ is given by
\[ \mathbb{E}_x \big( \int_0^{\infty} e^{-\lambda s} f(X_s) ds \big) = \int_0^{\infty} e^{-\lambda s} p_s^{\beta}f(x) ds=  r_{\lambda}^{\beta} f(x)=\int_{[0,\infty)} f(y)~ r^{\beta}_{\lambda}(x,dy).\]
The proof of Theorem \ref{thmdensity} is based on the so-called first passage time formula (see \cite[(6.4)]{Kni81}).\\
Let $\lambda >0$ and define $A^{\beta}:=\lambda - (r_{\lambda}^{\beta})^{-1}$ on $\mathcal{D}:=r_{\lambda}^{\beta} (C_0([0,\infty)))$ (which is independent of $\lambda$). By \cite[Theorem 6.2, Theorem 6.4]{Kni81} it holds that
\begin{align} \label{genC} A^{\beta} f= f^{\prime \prime}, \ \ \ f \in \mathcal{D}=\{ f \in C_0([0,\infty)) \cap C^2([0,\infty)|~ f^{\prime \prime} \in C_0([0,\infty)) \text{ and } \beta f^{\prime \prime}(0)=f^{\prime}(0) \}. \end{align}
The condition $\beta f^{\prime \prime}(0)=f^{\prime}(0)$ for $f \in C^2([0,\infty))$ is called \emph{Wentzell boundary condition}.

\begin{definition} \label{defDF}
Let $F$ be a locally compact separable metric space. A transition semigroup $p_t$, $t >0$, of an $F$-valued Markov process is said to have the {\em Feller property} if  $p_t (C_0(F)) \subset C_0(F)$ and $\lim_{t \downarrow 0} \Vert p_t f - f \Vert_{\infty}=0$ for each $f \in C_0(F)$. Furthermore, it is called {\em strong Feller} if $p_t(\mathcal{B}_b(F)) \subset C_b(F)$ for each $t >0$. If the transition semigroup has both Feller and strong Feller property, we say that it possesses the {\em doubly Feller property}.
\end{definition}

We can also deduce the following:
\begin{proposition} \label{1ddoublyfeller}
The transition semigroup $(p_t^{\beta})_{t >0}$ of sticky Brownian motion on $[0,\infty)$ has the doubly Feller property.
\end{proposition}

\begin{proof}
In consideration of Theorem \ref{thmdensity} it rests to show that $p^{\beta}_t(\mathcal{B}_b([0,\infty))) \subset C_b([0,\infty))$. Let $f \in \mathcal{B}_b([0,\infty))$ and $t >0$. By the inequality
\[ \frac{2}{\sqrt{\pi}} \int_x^{\infty} e^{-(z^2-x^2)} ~dz \leq \frac{2}{\sqrt{\pi}} \int_x^{\infty} e^{-(z-x)^2} ~dz=1  \]
it follows that $\text{erfc}(x) \leq e^{-x^2}$ for each $x \geq 0$. Let $x \in [0, \infty)$ and $(x_n)_{n \in \mathbb{N}}$ a sequence in $[0,\infty)$ such that $x_n \rightarrow x$ as $n \rightarrow \infty$. Then $G_n(y):=f(y) g_{0,\sqrt{2} \beta}(t,x_n +\frac{y}{\sqrt{2}})$ converges for each fixed $y \in [0,\infty)$ to $G(y):=f(y) g_{0,\sqrt{2} \beta}(t,x +\frac{y}{\sqrt{2}})$ as $n \rightarrow \infty$ by continuity of $g_{0,\sqrt{2} \beta}$ in the second variable. Moreover, for each $y \in [0,\infty)$ it holds
\begin{align*}
|G_n(y)| &\leq \Vert f \Vert_{\infty} K_1 \exp(\frac{\sqrt{2} x_n +y}{\beta}) \text{erfc}(\frac{x_n}{\sqrt{2t}}+\frac{y}{2\sqrt{t}}) \\
&\leq \Vert f \Vert_{\infty} K_2 \exp(\frac{y}{\beta}) \text{erfc}(\frac{y}{2 \sqrt{t}}) \\
&\leq \Vert f \Vert_{\infty} K_2 \exp(\frac{y}{\beta}) \exp(-\frac{y^2}{4t})=:H(y)
\end{align*}
for suitable constants $K_1$ and $K_2$. Note that the function $H$ is integrable with respect to the Lebesgue measure on $[0,\infty)$. Thus, dominated convergence yields
\[ \int_{[0,\infty)} G_n(y) dy \rightarrow \int_{[0,\infty)} G(y) dy \]
and by this, we can conclude that $p^{\beta}_tf$ is continuous and bounded. 
\end{proof}

\begin{remark}
Denote by $(T_t^{\beta})_{t \geq 0}$ the $L^2([0,\infty);m_{\beta})$-semigroup of $(\mathcal{E}^{\beta},D(\mathcal{E}^{\beta}))$ defined in (\ref{1dform}). Then, by the previous considerations,  for all $f \in \mathcal{B}_b([0,\infty)) \cap L^2([0,\infty);m_{\beta})$ it holds that $p_t^{\beta} f$ is a $\mu$-version of $T_t^{\beta} f$. Note also that the $L^2([0,\infty);m_{\beta})$-generator $(L,D(L))$ is given by
\[ Lf(x)= \mathbbm{1}_{(0,\infty)}(x) f^{\prime \prime}(x) + \mathbbm{1}_{\{0\}}(x) \frac{1}{\beta} f^{\prime}(x) \ \ \ \text{ for } f \in D(L)= H^{2,2}((0,\infty)),\]
where $H^{2,2}((0,\infty))$ denotes the Sobolev space of order two.
This can be shown using integration by parts, the fact that $D(\mathcal{E})=H^{1,2}((0,\infty))$ and the definition of the space $H^{2,2}((0,\infty))$. For $f \in C_c^2([0,\infty)) \subset D(L)$ such that the Wentzell boundary condition $\beta f^{\prime \prime}(0)=f^{\prime}(0)$ is fulfilled, it holds $Lf=f^{\prime \prime}$ similarly to the generator of the $C_0([0,\infty))$-semigroup given in (\ref{genC}). However, in the $L^2$-setting the boundary behavior is rather described by the measure $m_{\beta}$ instead of the domain of the generator.
\end{remark}

Next we will constuct the Dirichlet form corresponding to $n$ independent sticky Brownian motions on $[0,\infty)$, $n \in \mathbb{N}$. In \cite[Chapter V, Section 2.1]{BH91} it is shown how to construct finite tensor products of Dirichlet spaces. Moreover, the corresponding semigroup of the product Dirichlet form has an explicit representation. In our setting this construction yields the semigroup of an $n$-dimensional process on $E=[0,\infty)^n$, $n \in \mathbb{N}$, such that the components are independent sticky Brownian motions on $[0,\infty)$. In particular, this approach justifies the choice of the Dirichlet form structure used in \cite{FGV13}.\\

Let $(\mathcal{E}^{\beta},D(\mathcal{E}^{\beta}))$ be the Dirichlet form on $L^2([0,\infty);m_{\beta})$ defined in (\ref{1dform}). In accordance with \cite[Definition 2.1.1]{BH91} we define the product Dirichlet form $(\mathcal{E}^{n,\beta},D(\mathcal{E}^{n,\beta}))$ on $L^2([0,\infty)^n; m_{n,\beta})$ with $dm_{n,\beta}=\prod_{i=1}^n (dx_i + \beta \delta_0^i)$ by
\begin{small}
\begin{align} 
\mathcal{E}^{n,\beta}(f,g):=\sum_{i=1}^n \int_{[0,\infty)^{n-1}} \mathcal{E}^{\beta}(f(x_1,\dots, x_{i-1},\cdot,x_{i+1},\dots,x_n),g(x_1,\dots ,x_{i-1},&\cdot ,x_{i+1},\dots ,x_n)) \notag \\ 
&\prod_{j \neq i} (dx_j + \beta \delta_0^j) \label{ndform}
\end{align}
\end{small}
for $f,g \in D(\mathcal{E}^{n,\beta})$, where
\begin{align*} D(\mathcal{E}^{n,\beta}):=\{ &f \in L^2([0,\infty)^n;m_{n,\beta})\big|~\text{for each } i=1,\dots,n \text{ and for } \prod_{j \neq i}(dx_j+\beta \delta^j_0)-{a.e. } \\
&(x_1,\dots ,x_{i-1},x_{i+1},\dots ,x_n) \in [0,\infty)^{n-1}: f(x_1,\dots, x_{i-1},\cdot,x_{i+1},\dots,x_n) \in D(\mathcal{E}^{\beta}) \}
\end{align*}
First, we prove the following:

\begin{lemma} \label{lemdense}
$C_c^1([0,\infty)^n)$ is dense in $D(\mathcal{E}^{n,\beta})$.
\end{lemma}

\begin{proof}
Note that $C_c^1([0,\infty)^n) \subset D(\mathcal{E}^{n,\beta})$ by definition of $D(\mathcal{E}^{n,\beta})$.\\
W.l.o.g. let $n=2$. By \cite[Proposition 2.1.3 b)]{BH91} $D(\mathcal{E}^{\beta}) \otimes D(\mathcal{E}^{\beta})$ is dense in $D(\mathcal{E}^{2,\beta})$. We show that $C_c^1([0,\infty)) \otimes  C_c^1([0,\infty)) \subset C_c^1([0,\infty)^2)$ is dense in $D(\mathcal{E}^{\beta}) \otimes D(\mathcal{E}^{\beta})$. Then the assertion follows by a diagonal sequence argument. So let $h \in D(\mathcal{E}^{\beta}) \otimes D(\mathcal{E}^{\beta})$ such that $h(x_1,x_2)=f(x_1)g(x_2)$ for $m_{2,\beta}$-a.e. $(x_1,x_2) \in [0,\infty)^2$ and $f,g \in D(\mathcal{E}^{\beta})$. Choose sequences $(f_k)_{k \in \mathbb{N}}$ and $(g_k)_{k \in \mathbb{N}}$ in $C_c^1([0,\infty))$ such that $f_k \rightarrow f$ in $D(\mathcal{E}^{\beta})$ and $g_k \rightarrow g$ in $D(\mathcal{E}^{\beta})$ as $k \rightarrow \infty$ and define, for $k \in \mathbb{N}$, $h_k \in C_c^1([0,\infty)) \otimes  C_c^1([0,\infty))$ by $h_k(x_1,x_2):=f_k(x_1)g_k(x_2)$, $x_1,x_2 \in [0,\infty)$.
Then it follows immediately by assumption and the product structure that $h_k \rightarrow h$ as $k \rightarrow \infty$ in $L^2([0,\infty)^2; m_{2,\beta})$. Moreover, for $k,l \in \mathbb{N}$
\begin{align*}
\mathcal{E}^{2,\beta}(h_k-h_l)&=\int_{[0,\infty)} \mathcal{E}^{\beta}((h_k-h_l)(\cdot,x_2)) (dx_2 + \beta \delta_0^2) + \int_{[0,\infty)} \mathcal{E}^{\beta}((h_k-h_l)(x_1, \cdot)) (dx_1 + \beta \delta_0^1)\\
\leq &~\mathcal{E}^{\beta}(f_k-f_l)~ \Vert g_k \Vert_{L^2([0,\infty);dx+\beta \delta_0)} + \mathcal{E}^{\beta}(f_l)~ \Vert g_k - g_l \Vert_{L^2([0,\infty);dx+\beta \delta_0)} \\
&+ \mathcal{E}^{\beta}(g_k-g_l)~ \Vert f_k \Vert_{L^2([0,\infty);dx+\beta \delta_0)} + \mathcal{E}^{\beta}(g_l)~ \Vert f_k - f_l \Vert_{L^2([0,\infty);dx+\beta \delta_0)}.
\end{align*}
Hence, $\mathcal{E}^{2.\beta}(h_k-h_l) \rightarrow 0$ as $k,l \rightarrow \infty$ and thus, $h_k \rightarrow h$ as $k \rightarrow \infty$ in $D(\mathcal{E}^{2,\beta})$.
\end{proof}

Let $f,g \in C_c^1([0,\infty)^n)$. Then for each $i=1,\dots,n$ and fixed $(x_1,\dots ,x_{i-1},x_{i+1},\dots ,x_n) \in [0,\infty)^{n-1}$ we have
\begin{align*}
\mathcal{E}^{\beta}&(f(x_1,\dots, x_{i-1},\cdot,x_{i+1},\dots,x_n),g(x_1,\dots ,x_{i-1},\cdot ,x_{i+1},\dots ,x_n)) \\
&= \int_{[0,\infty)} \partial_i f(x_1,\dots ,x_n)~ \partial_i g(x_1,\dots ,x_n) ~dx_i.
\end{align*}
Set $\{ j \neq i \}:=\{1,\dots,i-1,i+1,\dots,n \}$. If $A$ is a subset of some set $I$, we denote by $A^c$ the set $I \backslash A$. Due to the identity
\[ \prod_{j \neq i} (dx_j + \beta \delta_0^j)= \sum_{A \subset \{j \neq i\}} \beta^{|A^c|} \prod_{j \in A} dx_j \prod_{j \in A^c} \delta_0^j \]
we get by rearranging the terms that 
\[ \mathcal{E}^{n,\beta}(f,g)= \sum_{\emptyset \neq B \subset \{1,\dots,n\}} \mathcal{E}_B(f,g) \]
with 
\[ \mathcal{E}_B(f,g):= \int_{[0,\infty)^n} \sum_{i \in B} \partial_i f ~ \partial_i g~d\lambda_B^{n,\beta}, \]
where $d\lambda_B^{n,\beta}:=\beta^{n-|B|} \prod_{j \in B} dx_j \prod_{j \in B^c} \delta_0^j$. In other words, $(\mathcal{E}^{n,\beta},D(\mathcal{E}^{n,\beta}))$ defined in (\ref{ndform}) coincides with the form defined in \cite[(2.3)]{FGV13} disregarding that in our present setting the density function $\varrho$ is identically one. Moreover, (\ref{ndform}) can also be rewritten in the form
\[ \mathcal{E}^{n,\beta}(f,g)= \int_E \sum_{i=1}^n \mathbbm{1}_{\{ x_i \neq 0 \}} ~ \partial_i f ~\partial_i g ~ dm_{n,\beta} \ \ \ \text{ for } f,g \in C_c^1(E). \] 
From the present point of view $(\mathcal{E}^{n,\beta},D(\mathcal{E}^{n,\beta}))$, defined as in (\ref{ndform}), is the sum of $n$ subforms and each such form for $i=1,\dots,n$ describes the dynamics of the process on $[0,\infty)^n$ for all configurations where the $i$-th component is not pinned to zero. In contrast, the forms $\mathcal{E}_B$, $\emptyset \neq B \subset \{1,\dots,n\}$ describe the dynamics of the process for all configurations where exactly the components specified by $B$ are non-zero.\\

By a minor generalization of the results in \cite{FGV13} we get the following lemma:

\begin{lemma}
The Dirichlet form $(\mathcal{E}^{n,\beta},D(\mathcal{E}^{n,\beta}))$ on $L^2([0,\infty)^n;m_{n,\beta})$, $n \in \mathbb{N}$, is conservative, strongly local, strongly regular and symmetric.
\end{lemma}

Let $x=(x_1,\dots,x_n),y=(y_1,\dots,y_n) \in [0,\infty)^n$, $n \in \mathbb{N}$. Then the transition kernel $p_t^{n,\beta}(x,dy)$ of $n$ independent sticky Brownian motions on $[0,\infty)$ is given by
\[ p_t^{n,\beta}(x,dy)=\prod_{i=1}^n p_t^{\beta}(x_i,dy_i). \]
Thus, for $f \in C_0([0,\infty)^n)$ we have
\[ p_t^{n,\beta}f(x)=\int_{[0,\infty)^n} f(y_1,\dots,y_n) \prod_{i=1}^n p_t^{\beta}(x_i,dy_i). \]
With the definition 
\[ \hat{p}_t^{\beta,i}f(x):=p_t^{\beta}f(x_1,\dots,x_{i-1},\cdot,x_{i+1},\dots,x_n)(x_i) \]
this can be rewritten in the form 
\[ p_t^{n,\beta}f(x)=\hat{p}_t^{\beta,n} \cdots \hat{p}_t^{\beta,1}f(x).\]
By Theorem \ref{thmdensity} we have an explicit representation of $p_t^{n,\beta}(x,dy)$ and presumably it is feasible to use this representation in order to deduce the doubly Feller property for $p_t^{n,\beta}$ similar to the one dimensional case. Instead, we use a general approach which shows that the strong Feller property transfers from $(p_t^{\beta})_{t >0}$ to the product semigroup. For this purpose, the following lemma is useful which relies on Deni's lemma and can be found in \cite[Chapter 1, Lemma 5.11]{Rev84} (also in the context of strong Feller properties of product kernels):

\begin{lemma} \label{lemmaRevuz}
Let $(E,d)$ be a locally compact separable metric space and $\kappa$ a sub-Markovian kernel on $(E,\mathcal{B}(E))$. Moreover, assume that $(g_k)_{k \in \mathbb{N}}$ is a bounded sequence in $\mathcal{B}_b(E)$ with pointwise limit $g$. Then the sequence $(\kappa g_k)_{k \in \mathbb{N}}$ converges locally uniformly to $\kappa g$.
\end{lemma}

\begin{proposition}
The transition semigroup $(p_t^{n,\beta})_{t >0}$ of $n$ independent sticky Brownian motions on $[0,\infty)$ has the doubly Feller property.
\end{proposition}

\begin{proof}
For simplicity assume again $n=2$. By Theorem \ref{thmdensity} $(p_t^{\beta})_{t >0}$ is a conservative Feller semigroup and hence, it extends to a conservative Feller semigroup $(\tilde{p}_t^{\beta})_{t >0}$ on the one point compactification $[0,\infty)^{\Delta}:=[0,\infty) \cup \{ \Delta \}$ as stated in \cite[Lemma 17.13]{Kal97}. This corresponds to an extension of the kernels $p_t^{\beta}(x,dy)$, $t >0$, $x \in [0,\infty)$, from $([0,\infty),\mathcal{B}([0,\infty))$ to $([0,\infty)^{\Delta},\mathcal{B}([0,\infty)^{\Delta}))$ by setting
\[\tilde{p}^{\beta}_t(x,A):=\left\{\begin{array}{cl} p^{\beta}_t(x,A \backslash \{ \Delta \}), & \mbox{if } x \in [0,\infty) \\ 
\delta_{\Delta}(A), & \mbox{if } x=\Delta \end{array}\right.
                                       \]
for $A \in \mathcal{B}([0,\infty)^{\Delta})$. Hence, the argument of \cite[Lemma 20.16]{Kal97} can be applied in use of \cite[Theorem 3.29, Lemma 17.3]{Kal97} (note that $[0,\infty)^{\Delta}$ is metrizable, since $[0,\infty)$ is locally compact, second countable and Hausdorff) and the transition semigroup of two independent sticky Brownian motions on $[0,\infty)$ gives rise to a Feller semigroup $(\tilde{p}_t^{2, \beta})_{t >0}$ on $C([0,\infty)^{\Delta} \times [0,\infty)^{\Delta})$ of two independent processes. We use the one point compactification, since the used results are formulated for compact metric spaces. Let $f \in C_0([0,\infty)^2)$. Then, $f$ extends uniquely to an element $\tilde{f} \in C([0,\infty)^{\Delta} \times [0,\infty)^{\Delta})$ by setting $\tilde{f}(x,y):=0$, whenever $x=\Delta$ or $y=\Delta$. This definition yields 
\[ \tilde{p}^{2,\beta}_t \tilde{f}(x_1,x_2)=p^{2,\beta}_t f(x_1,x_2) \text{ for } (x_1,x_2) \in [0,\infty)^2 \ \text{ and } \tilde{p}^{2,\beta}_t \tilde{f}(x_1,x_2)=0 \text{ else}.\]
Consequently, the restriction of $\tilde{p}^{2,\beta}_t \tilde{f}$ to $[0,\infty)^2$ and thus, $p^{2,\beta}_t f$ is an element of $C_0([0,\infty)^2)$. Moreover, the strong continuity transfers directly from $(\tilde{p}^{2,\beta}_t)_{t >0}$ to $(p^{2,\beta}_t)_{t >0}$.\\
Let $t >0$, $f \in \mathcal{B}_b([0,\infty)^2)$, $(x_1,x_2) \in [0,\infty)^2$ and $((x_1^k,x_2^k))_{k \in \mathbb{N}}$ a sequence in $[0,\infty)^2$ converging to $(x_1,x_2)$. Then, $\hat{p}_t^{\beta,1}f(x_1^k,\cdot)$, $k \in \mathbb{N}$, yields a bounded sequence in $\mathcal{B}_b([0,\infty))$ due to the boundedness of $f$ and the contraction property. Moreover, for fixed $y \in [0,\infty)$ it holds $\hat{p}_t^{\beta,1}f(x_1^k,y) \rightarrow \hat{p}_t^{\beta,1}f(x_1,y)$ as $k \rightarrow \infty$ by the strong Feller property of $(p_t^{\beta})_{t >0}$. Due to Lemma \ref{lemmaRevuz} we can conclude that the sequence given by $p_t^{\beta}$ applied to $\hat{p}_t^{\beta,1}f(x_1^k,\cdot)$ converges locally uniformly to $p_t^{\beta}$ applied to $\hat{p}_t^{\beta,1}f(x_1,\cdot)$. Hence, 
\[ p_t^{2,\beta}f(x_1^k,x_2^k)=\hat{p}_t^{\beta,2}\hat{p}_t^{\beta,1}f(x_1^k,x_2^k) \rightarrow \hat{p}_t^{\beta,2}\hat{p}_t^{\beta,1}f(x_1,x_2)=p_t^{2,\beta}f(x_1,x_2) \ \text{ as } k \rightarrow \infty \]
and thus, $p_t^{2,\beta}(\mathcal{B}_b([0,\infty)^2)) \subset C_b([0,\infty)^2)$.
\end{proof}

Let $(T_t^{\beta})_{t \geq 0}$ be the $L^2([0,\infty);m_{\beta})$-semigroup of $(\mathcal{E}^{\beta},D(\mathcal{E}^{\beta}))$. For $f \in L^2([0,\infty)^n;m_{n,\beta})$, $i=1,\dots,n$, and $m_{n,\beta}$-a.e. $(x_1,\dots,x_n) \in [0,\infty)^n$ set
\[ \hat{T}_t^{\beta,i} f(x_1,\dots,x_n):= T_t^{\beta} f(x_1,\dots,x_{i-1},\cdot,x_{i+1},\dots,x_n)(x_i). \]
and
\[ T_t^{n,\beta} f= \hat{T}_t^{\beta,1} \cdots \hat{T}_t^{\beta,n} f.\]
By \cite[Proposition 2.1.3 a)]{BH91} $(T_t^{n,\beta})_{t \geq 0}$ is the $L^2([0,\infty)^n;m_{n,\beta})$-semigroup associated to the form $(\mathcal{E}^{n,\beta},D(\mathcal{E}^{n,\beta}))$ defined in (\ref{ndform}) and the order of the $\hat{T}_t ^{\beta,i}$, $i=1,\dots,n$, is arbitrary.\\ 
Let $f \in \mathcal{B}_b([0,\infty)^n) \cap L^2([0,\infty)^n;m_{n,\beta})$. Then we have for $m_{n,\beta}$-a.e. $x=(x_1,\dots,x_n) \in [0,\infty)^n$
\begin{align} \label{iter1}
\hat{T}_t^{\beta,n} f(x_1,\dots,x_n)=T_t^{\beta} f(x_1,\dots,x_{n-1},\cdot)(x_n) &=p_t^{\beta} f(x_1,\dots,x_{n-1},\cdot)(x_n) \\
&=\int_{[0,\infty)} f(x_1,\dots,x_{n-1},y_n) p^{\beta}_t(x_n,dy_n) \notag
\end{align}
and similarly
\begin{align} \label{iter2}
\hat{T}_t^{\beta,n-1} \hat{T}_t^{\beta,n} f(x_1,\dots,x_n)= \int_{[0,\infty)} \int_{[0,\infty)} f(x_1,\dots,x_{n-2},y_{n-1},y_n) p^{\beta}_t(x_n,dy_n) p^{\beta}_t(x_{n-1},dy_{n-1}).
\end{align}

Proceeding successively as in (\ref{iter1}) and (\ref{iter2}), together with the preceding considerations, proves Proposition \ref{propindep}.
 
\subsection{Girsanov transformations} \label{secgirsanov}

We summerize some results on Girsanov transformations of a Markov process and the associated Dirichlet form. The statements can be found in \cite{Ebe96}, \cite{Fit08} and \cite[Chapter 6]{FOT94}. In some cases we do not state the results in full generality, since  for our purposes it is sufficient to simplify the assumptions.\\

Let $\mathbb{M}=(\Omega,\mathcal{F},(\mathcal{F}_t)_{t \geq 0}, (X_t)_{t \geq 0}, (\mathbb{P}_x)_{x \in F})$ be a $\mu$-symmetric strong Markov process with state space $F \subset \mathbb{R}^n$, $n \in \mathbb{N}$, continuous sample paths and infinite lifetime, where $\mu$ is a positive Radon measure on $(F,\mathcal{B}(F))$ with full support. We suppose that the process is canonical, i.e., $\Omega=C([0,\infty),F)$ and $X_t(\omega)=\omega(t)$ for $\omega \in \Omega$ and $t \geq 0$. Moreover, assume that its Dirichlet form $(\mathcal{E},D(\mathcal{E}))$ on $L^2(F;\mu)$ is strongly regular (property (D1)), strongly local, conservative and that it possesses a square field operator $\Gamma$ (property (D2)). We denote its generator by $(L,D(L))$. Suppose that $\mathcal{D}:=C_c^1(F)$ a dense subspace of $D(\mathcal{E})$, $D(L) \cap \mathcal{D}$ is dense in $D(\mathcal{E})$ and for every $f \in \mathcal{D}$ it holds $f,\Gamma(f) \in L^{\infty}(F;\mu)$. Denote by $(p_t)_{ t>0}$ the transition semigroup of $\mathbb{M}$, i.e., for $f \in \mathcal{B}_b(F)$ it holds
\[ p_tf(x):= \mathbb{E}_x(f(X_t)),\]
and we suppose that the transition density $p_t(x,\cdot)$, $x \in F$, $t >0$, possesses the absolute continuity condition \cite[(4.2.9)]{FOT94}. \\
A function $f$ is said to be in $D(\mathcal{E})_{\text{loc}}$ (respectively $D(\mathcal{E})_{b,\text{loc}}$) if for any relatively compact open set $G$ there exists a function (respectively bounded function) $g \in D(\mathcal{E})$ such that $f=g$ $\mu$-a.e. on $G$. Fix some $\phi \in D(\mathcal{E})_{b,\text{loc}}  \cap C_b(F)$ such that $\phi$ is strictly positive and $\ln \phi \in D(\mathcal{E})_{\text{loc}}$. Define $\varrho:=\phi^2$ and the symmetric bilinear form $(\mathcal{E}^{\varrho},\mathcal{D}^{\varrho})$ by
\begin{align} &\mathcal{D}^{\varrho}:=\{ f \in D(\mathcal{E})|~\int_F (\Gamma(f)+f^2) \varrho d\mu < \infty \}, \\
&\mathcal{E}^{\varrho}(f,g):= \int_F \Gamma(f,g)~ \varrho d\mu. \notag
\end{align}
In particular, $\mathcal{D}^{\varrho}=D(\mathcal{E})$, since $\varrho$ is bounded. \\

Under the above assumptions the conditions (D1)-(D3) of \cite{Ebe96} are fulfilled. The rather general property (D3) holds, since the form $(\mathcal{E},D(\mathcal{E}))$ is required to be strongly regular and $D(L) \cap \mathcal{D}$ is assumed to be dense in $D(\mathcal{E})$. Moreover, $\varrho$ is (locally) bounded and $\ln \phi \in D(\mathcal{E})_{\text{loc}}$. Thus, by \cite[Theorem 1.1, Corollary 1.3]{Ebe96} we can conclude the following:

\begin{lemma} \label{lemclosureGirsanov} The symmetric bilinear form $(\mathcal{E}^{\varrho},D(\mathcal{E}))$ is densely defined and closable on \linebreak $L^2(F;\varrho \mu)$ and its closure $(\mathcal{E}^{\varrho}, D(\mathcal{E}^{\varrho}))$ is a strongly local Dirichlet form.
Moreover, $(\mathcal{E}^{\varrho}, D(\mathcal{E}^{\varrho}))=(\mathcal{E}^{\varrho},\overline{\mathcal{D}}))$, i.e., $\mathcal{D}$ is a dense subset of $D(\mathcal{E}^{\varrho})$.
\end{lemma}

Due to \cite[Theorem 5.5.1]{FOT94} it is possible to give a Fukushima decomposition of the process $\mathbb{M}$ of the form
\begin{align} \label{fukudec} \ln \phi (X_t) - \ln \phi(X_0)= M_t^{[\ln \phi]} + N_t^{[\ln \phi]} \ \ \mathbb{P}_x-\text{a.s. for each } x \in F,~t \geq 0, 
\end{align}
where $(M_t^{[\ln \phi]})_{t \geq 0}$ is a martingale additive functional of locally finite energy and $(N_t^{[\ln \phi]})_{t \geq 0}$ is a continuous additive functional of locally zero energy (see \cite[p.273]{FOT94}). The decomposition holds indeed for every starting point in $F$, since the transition semigroup of $\mathbb{M}$ fulfills the absolute continuity condition by the above assumptions. Note that the construction of the decomposition (\ref{fukudec}) requires some localization argument using the definition of $D(\mathcal{E})_{\text{loc}}$. In view of \cite[Theorem 5.5.5]{FOT94} it is even possible to determine $(M_t^{[\ln \phi]})_{t \geq 0}$ and $(N_t^{[\ln \phi]})_{t \geq 0}$ more explicitly in more specific cases. \\

Define the positive multiplicative functional $(Z_t)_{t \geq 0}$ by
\begin{align} \label{MF} Z_t=\exp( M_t^{[\ln \phi]} - \frac{1}{2} \langle M^{[\ln \phi]} \rangle_t),~t \geq 0. 
\end{align}
Then, it is possible to define a $\varrho \mu$-symmetric diffusion \[\mathbb{M}^{\varrho}:=(\Omega,\mathcal{F},(\mathcal{F}_t)_{t \geq 0}, (X_t)_{t \geq 0}, (\mathbb{P}^{\varrho}_x)_{x \in F})\] by the formula 
\[ {\mathbb{P}^{\varrho}_x}_{|\mathcal{F}_t} := Z_t ~{\mathbb{P}_x}_{|\mathcal{F}_t},~x \in F, \]
as stated in \cite[Section 1]{Fit08}. Thus, $\mathbb{P}^{\varrho}_x$ is locally absolutely continuous with respect to $\mathbb{P}_x$ for $x \in F$ and the transition semigroup $(p^{\varrho}_t)_{t > 0}$ of $\mathbb{M}^{\varrho}$ is given by
\[ p^{\varrho}_t f(x)= \mathbb{E}_x(Z_t f(X_t)),~t \geq 0, \]
for $f \in \mathcal{B}_b(F)$. Let $(G_k)_{k \in \mathbb{N}}$ be an increasing sequence of open, relatively compact subsets of $F$ such that $F=\cup_{k \in \mathbb{N}} G_k$. In \cite[Section 1]{Fit08} it is shown that the results of \cite{Fit97} apply to the present setting and the Dirichlet form of $\mathbb{M}^{\varrho}$ is given by the closure of $(\mathcal{E}^{\varrho},\mathcal{C})$ on $L^2(F;\varrho \mu)$, where
\[ \mathcal{C}:= \bigcup_{k \in \mathbb{N}} \{ f \in D(\mathcal{E}) \cap L^{\infty}(F;\mu)|~\tilde{f}=0 ~\text{ q.e. on } F \backslash G_k\}. \]
Hence, the Dirichlet form of $\mathbb{M}^{\varrho}$ is $(\mathcal{E}^{\varrho},D(\mathcal{E}^{\varrho}))$ in view of Lemma \ref{lemclosureGirsanov}, since $\mathcal{D} \subset \mathcal{C} \subset D(\mathcal{E})$. \\
We say that the process $\mathbb{M}^{\varrho}$ and the transition semigroup $(p^{\varrho}_t)_{t >0}$ are the Girsanov transformation of $\mathbb{M}$ and $(p_t)_{t >0}$ respectively by the multiplicative functional $(Z_t)_{t \geq 0}$.

\section{Construction of the strong Feller transition semigroup} \label{Feller}

In \cite{Chu85} criteria are given under which the doubly Feller property is preserved under the transformation by a multiplicative functional $(Z_t)_{t \geq 0}$. This concept is extended in \cite{CK08}. It is shown that the conditions on $(Z_t)_{t \geq 0}$ can be weakened. Moreover, the setting is applied to Feynman-Kac and Girsanov transformations. In particular, precise conditions on the Revuz measure of the underlying additive functionals are given. We quote a result of \cite{CK08} concerning the preservation of the doubly feller property under Girsanov transformations. Since we deal with strong Markov processes with {\em continuous sample paths}, we restrict the results to this setting instead of stating them in full generality.\\

Let $\mathbb{M}=(\Omega,\mathcal{F},(\mathcal{F}_t)_{t \geq 0}, (X_t)_{t \geq 0}, (\mathbb{P}_x)_{x \in F})$ be again a $\mu$-symmetric strong Markov process with state space $F \subset \mathbb{R}^n$, $n \in \mathbb{N}$, continuous sample paths and infinite lifetime, where $\mu$ is a positive Radon measure on $(F,\mathcal{B}(F))$ with full support. As before, denote by $(p_t)_{ t>0}$ the transition semigroup of $\mathbb{M}$. Assume that $(p_t)_{ t>0}$ possesses the doubly Feller property. \\
Let $r_{\lambda}(x,y)$, $\lambda >0$, $x,y \in F$, be the resolvent kernel of $\mathbb{M}$, i.e., the resolvent $(r_{\lambda})_{\lambda>0}$ of $\mathbb{M}$ is given by
\[ r_{\lambda}f(x)=\int_F f(y) r_{\lambda}(x,y) d\mu(y) \]
for $f \in \mathcal{B}_b(F)$, $\lambda >0$ and $x \in F$. For a Borel measure $\nu$ on $\mathcal{B}(F)$ we define the $\lambda$-\textit{potential of} $\nu$ by $R_{\lambda}\nu(x):= \int_F r_{\lambda}(x,y) d\nu(y)$, $\lambda >0$.\\
Let $B$ be a non-empty open subset of $F$ and denote by $B_{\Delta}:=B \cup \{\Delta\}$ the one-point compactification of $B$. Define $(X_t^B)_{ t \geq 0}$ by
\begin{align*}
X_t^B:=
\left\{
\begin{array}{l}
X_t \ \ \text{ if } t < \tau_B \\
\Delta  \ \text{ if } t \geq \tau_B
\end{array}
\right.
\end{align*} 
where $\tau_B :=\inf \{ t>0 |~ X_t \notin B \}$.
The transition semigroup of $(X_t^B)_{ t \geq 0}$ is given by
\[ p_t^B(x,A)=\mathbb{P}_x(X_t \in A,~ t < \tau_B) \]
and 
\[ p_t^B(x,\{\Delta\}):=1-p_t^B(x,B), \ \ p_t^B(\Delta, \{ \Delta \}):=1, \]
for $x \in B$, $A \in  \mathcal{B}(B)$. A function $f \in \mathcal{B}_b(F)$ is extended to $\Delta$ by setting $f(\Delta)=0$. For functions of this form, the transition semigroup of $(X_t^B)_{ t \geq 0}$ reads
\[ p_t^B f(x)=\mathbb{E}_x (f(X_t) \mathbbm{1}_{\{t < \tau_B\}}). \]
The set $B$ is called \textit{regular} if for each $x \in F \backslash B$, we have $\mathbb{P}_x(\tau_B=0)=1$.\\

Let $(M_t)_{t \geq 0}$ be a continuous locally square integrable martingale additive functional and denote by $\mu_{\langle M \rangle}$ the Revuz measure of $(\langle M \rangle_t)_{t \geq 0}$. Furthermore, the transition semigroup $(\tilde{p}_t^B)_{t \geq 0}$ is given by
\[ \tilde{p}_t^Bf(x):=\mathbb{E}_x(Z_t f(X_t) \mathbbm{1}_{\{t < \tau_B\}}), \]
where $Z_t:=\exp (M_t-\frac{1}{2} \langle M \rangle_t )$, $t \geq 0$ and corresponds to the process obtained from $\mathbb{M}^{\varrho}$ (see Section \ref{secgirsanov}) killed when leaving $B$. In the special case $B=F$ this definition reduces to the transition semigroup of $\mathbb{M}^{\varrho}$. 

\begin{definition}
A Borel measure $\nu$ on $\mathcal{B}(F)$ is said to be of 
\begin{enumerate}
\item[(i)] {\em Kato class} if $\lim_{\lambda \rightarrow \infty} \sup_{x \in F} R_{\lambda} \nu(x)=0$,
\item[(ii)] {\em extended Kato class} if $\lim_{\lambda \rightarrow \infty} \sup_{x \in F} R_{\lambda} \nu(x) < 1$,
\item[(iii)] {\em local Kato class} if $\mathbbm{1}_K \nu$ is of Kato class for every compact set $K \subset F$. 
\end{enumerate}
\end{definition}

\begin{theorem} \label{thmGT}
Assume that $\frac{1}{2} \mu_{\langle M \rangle}$ is a positive Radon measure of local and extended Kato class and let $B$ be a regular open subset of $F$. Then $(\tilde{p}_t^B)_{t \geq 0}$ has the doubly Feller property. Moreover, $(Z_t)_{t \geq 0}$ is a martingale and
\begin{align*}
&\lim_{t \rightarrow 0} \sup_{x \in D} ~\mathbb{E}_x( |Z_t-1| \mathbbm{1}_{\{t < \tau_D\}}) =0 \text{ for any relatively compact open set } D \subset B, \\
& \sup_{0 \leq s \leq t}~ \sup_{x \in B} ~\mathbb{E}_x ( Z_s^p \mathbbm{1}_{\{s < \tau_B \}} ) < \infty \text{ for some } p>1 \text{ and each } t>0.
\end{align*}
\end{theorem}

\begin{proof}
See \cite[Theorem 3.3]{CK08}.
\end{proof}

Let
\[ \mathbb{M}^{n,\beta}=(\Omega,\mathcal{F},(\mathcal{F}_t)_{t \geq 0}, (X_t)_{t \geq 0}, (\mathbb{P}^{n,\beta}_x)_{x \in E}) \]
be the process of Proposition \ref{propindep} with doubly Feller transition semigroup $(p_t^{n,\beta})_{t >0}$ and Dirichlet form $(\mathcal{E}^{n,\beta},D(\mathcal{E}^{n,\beta}))$ on $L^2(E;m_{n,\beta})$. Denote by $(r_{\lambda}^{n,\beta})_{\lambda>0}$ the resolvent, by $r_{\lambda}^{n,\beta}(x,y)$, $x,y \in E$, $\lambda >0$, its density with respect to $m_{n,\beta}$, by $R_{\lambda}^{n,\beta} \nu$ the $\lambda$-potential of some Borel measure $\nu$ for $\lambda>0$ and by $\mathbb{E}^{n,\beta}_x$ the expectation with respect to $\mathbb{P}^{n,\beta}_x$, $x \in E$. Similarly as before we also write $R_{\lambda}^{\beta}$ instead of $R_{\lambda}^{1,\beta}$ etc. in the case $n=1$. In the following, we introduce a density function $\varrho=\phi^2$. Under suitable conditions on $\phi$ it is possible to perform a Girsanov transformation such that the transition semigroup of the transformed process $\mathbb{M}^{n,\beta,\varrho}$ still possesses the strong Feller property (or even the doubly Feller property). By the preceding section the transformed Dirichlet form is of the form considered in \cite{FGV13}. In this way, we are able to strengthen the results in \cite{FGV13}.

\begin{remark}
For functions $\phi$ such that the conditions of Theorem \ref{thmGT} are fulfilled for $(Z_t)_{t \geq 0}$ as in (\ref{MF}) and $B=E$, we immediately get that the transition function has the doubly Feller property and the process $\mathbb{M}^{n,\beta,\varrho}$ solves (\ref{sde!}) for every starting point in $E$. Unfortunately, we are also interested in densities $\varrho$ such that the corresponding Revuz measure is not of extended Kato class. Such potentials are of particular interest for the application to the so-called wetting model in the theory of stochastic interface models. For this reason, we construct a strong Feller transition semigroup for a larger class of densities using Theorem \ref{thmGT} and an approximation argument. A direct application fails, since the Kato condition on $\mu_{\langle M \rangle}$ ensures that the drift caused by the Girsanov transformation does not "explode". However, this criterion does only take into account the variation of the drift, but not its direction, which is of particular importance in our setting.
\end{remark}

\begin{example} \label{example}
Let $n=1$ and $\phi(x):=\exp(-\frac{1}{4} x^2)$. In this case, $(\ln \phi)^{\prime}(x)=-\frac{1}{2}x$. Hence, we expect that the process $\mathbb{M}^{1,\beta,\phi^2}$ has the representation
\[ dX_t= \sqrt{2}~ \mathbbm{1}_{(0,\infty)}(X_t) dB_t - X_t~ \mathbbm{1}_{(0,\infty)}(X_t) dt + \frac{1}{\beta} \mathbbm{1}_{\{0\}}(X_t) dt. \]
Note that the additional drift term is always non-positive, since $X_t \in [0,\infty)$ for all $t >0$ and thus, it attracts the process to $0$. However, the logarithmic derivative of $\phi$ is unbounded and the energy measure is even not of extended Kato class. Indeed, 
\begin{align*}
R_{\lambda} \mu_{\langle \ln \phi \rangle}(x) &= \int_{[0,\infty)} r^{\beta}_{\lambda}(x,y) d\mu_{\langle \ln \phi \rangle}(y) \\
&=\int_{[0,\infty)} \big( \frac{1}{\sqrt{2\lambda}} (e^{-\sqrt{2\lambda}|x-y|}-e^{-\sqrt{2\lambda}(x+y)}) + \frac{1}{\sqrt{2 \lambda}+\sqrt{2}\beta \lambda} 2 e^{-\sqrt{2 \lambda}(x+y)} \big)~ y^2 dy
\end{align*}
is unbounded in $x$ for each fixed $\lambda >0$, since
\[ \int_{[0,\infty)} \frac{1}{\sqrt{2\lambda}}  e^{-\sqrt{2\lambda}|x-y|}~ y^2 dy = \frac{1}{\lambda} x^2 - \frac{1}{2 \lambda^2} e^{-\sqrt{2 \lambda} x} + \frac{1}{\lambda^2} \rightarrow \infty \ \text{ as } x \rightarrow \infty,
\]
whereas the remaining terms converge to $0$ as $x \rightarrow \infty$. Thus, it is not possible to apply Theorem \ref{thmGT} to this specific choice of $\phi$.
\end{example}

Assume that $\phi$ is given such that Condition \ref{conditions} is fulfilled and additionally $\phi \in C^2(E)$. This additional condition will also be assumed later on. Then, It{\^o}'s formula yields 
\begin{align} \label{AF}
M^{[\ln \phi]}_t= \sqrt{2}~ \sum_{i=1}^n \int_0^t \partial_i \ln \phi(X_s)~ \mathbbm{1}_{(0,\infty)}(X_s^i) dB_s^i,~t \geq 0.
\end{align}
If we do not assume the additional condition $\phi \in C^2(E)$ the same representation follows by an approximation argument. Fix some relatively compact, open subset $G$ of $E$ and let $u \in C^1_c(E) \subset D(\mathcal{E}^{n,\beta})$ be given such that $u=\ln \phi$ on $G$. Moreover, choose a sequence $(u_k)_{k \in \mathbb{N}}$ in $C_c^2(E)$ such that $u_k \rightarrow u$ as $k \rightarrow \infty$ in $C^1$-norm. In particular, the convergence holds with respect to the $\mathcal{E}^{n,\beta}_1$-norm. Then, $M_t^{[\ln \phi]}=M_t^{[u]}$ for $t < \tau_G=\inf \{ t >0|~X_t \notin G\}$, $M_t^{[u]}=\lim_{k \rightarrow \infty} M_t^{[u_k]}$ a.s. for $t \geq 0$ and 
\begin{align*}
M^{[u_k]}_t= \sqrt{2}~ \sum_{i=1}^n \int_0^t \partial_i u_k(X_s)~ \mathbbm{1}_{(0,\infty)}(X_s^i) dB_s^i,~t \geq 0.
\end{align*}
Furthermore, due to the uniform convergence of the sequence $(\partial_i u_k)_{k \in \mathbb{N}}$ and the It{\^o} isometry it follows that
\begin{align*} &\mathbb{E}^{n,\beta}_x \Big( \sum_{i=1}^n \big(  \int_0^t \partial_i u_k(X_s)~ \mathbbm{1}_{(0,\infty)}(X_s^i) dB_s^i - \int_0^t \partial_i u(X_s)~ \mathbbm{1}_{(0,\infty)}(X_s^i) dB_s^i \big)^2 \Big) \\
=&\mathbb{E}^{n,\beta}_x \Big( \sum_{i=1}^n \big( \int_0^t (\partial_i u_k - \partial_i u)(X_s)~ \mathbbm{1}_{(0,\infty)}(X_s^i) dB_s^i \big)^2 \Big) \\
\leq& \mathbb{E}^{n,\beta}_x \Big( \sum_{i=1}^n \int_0^t (\partial_i u_k - \partial_i u)^2(X_s)~ ds \Big) \rightarrow 0 \text{ as } k \rightarrow \infty.
\end{align*}
In particular, it holds a.s. convergence for a subsequence and hence, using $u=\ln \phi$ on $G$ we get 
\[ M_t^{[\ln \phi]}=M_t^{[u]}=\sqrt{2}~\sum_{i=1}^n \int_0^t \partial_i \ln \phi (X_s)~ \mathbbm{1}_{(0,\infty)}(X_s^i) dB_s^i, ~t < \tau_G, \]
$\mathbb{P}^{n,\beta}_x$-a.s. for every $x \in E$.

\begin{example} \label{exbounded} Let $\grad \ln \phi$ additionally be essentially bounded w.r.t. $m_{n,\beta}$. Then $\frac{1}{2} \mu_{\langle \ln \phi \rangle}=\frac{1}{2} \mu_{\langle M^{[\ln \phi]} \rangle}$ is of local and extended Kato class.
\end{example}

Let $k \in \mathbb{N}$ and $K:=[0,k)^n$ as well as $\tau_k:=\inf \{ t >0 |~ X_t \notin K \}$. Let $\phi_k$  be given such that $\phi_k=\phi$ on $K$, Condition \ref{conditions} is fulfilled for $\phi_k$ and $\nabla \ln \phi_k \in L^{\infty}(E;m_{n,\beta})$. We define the exponential functional $(Z^k_t)_{t \geq 0}$ by 
\[ Z^k_t := \exp(M^{[\ln \phi_k]}_t - \frac{1}{2} \langle M^{[\ln \phi_k]} \rangle_t ).\] Note that we are in fact only interested in the restriction of $\phi$ to the set $K$, since the function is used to define a Girsanov transformation of $(p_t^{n,\beta})_{t >0}$ which is killed when leaving $K$. Nevertheless, in order to give meaning to $Z_t^k$ for $t \geq \tau_k$, we extend $\phi_k$ to $E$.

\begin{theorem} \label{strongfeller}
Let $\varrho=\phi^2$ be given as in Condition \ref{conditions} and $Z_t=\exp(M_t^{[\ln \phi]}- \frac{1}{2} \langle M^{[\ln \phi]} \rangle_t)$, $t \geq 0$. Then the transition function $(p^{n,\beta,\varrho}_t)_{t \geq 0}$ defined by $p^{n,\beta,\varrho}_tf(x)=\mathbb{E}^{n,\beta}_x(Z_tf(X_t))$ for $f \in \mathcal{B}_b(E)$ and $x \in E$ which corresponds to the strong Markov process $\mathbb{M}^{n,\beta,\varrho}$ has the strong Feller property.
\end{theorem}

\begin{proof}
$K$ is regular, i.e., $\mathbb{P}^{n,\beta}_x(\tau_k =0)=1$ for each $x \in E \backslash K$, since $\mathbb{M}^{n,\beta}$ has continuous sample paths. We define the transition function $(p^k_t)_{t \geq 0}$ similar as $(p^{n,\beta,\varrho}_t)_{t \geq 0}$ by $p_t^k f(x):=\mathbb{E}^{n,\beta}_x(Z_t^k f(X_t) \mathbbm{1}_{\{t <\tau_k\}})$. By the assumptions on $\phi_k$, Example \ref{exbounded} and Theorem \ref{thmGT}, $(p_t^k)_{t \geq 0}$ has the doubly Feller property for each $k>0$. Let $f \in \mathcal{B}_b(E)$ and choose a constant $C(f) < \infty$ such that $|f(x)| \leq C(f)$ for all $x \in E$. Clearly, $p^{n,\beta,\varrho}_tf \in \mathcal{B}_b(E)$. Hence, it suffices to show the continuity of $p^{n,\beta,\varrho}_tf$. We have for $x \in D:=[0,d]^n$, $d>0$,
\begin{align*}
| p^{n,\beta,\varrho}_tf(x) - p_t^kf(x)| &=| \mathbb{E}^{n,\beta}_x(Z_t f(X_t)) -\mathbb{E}^{n,\beta}_x(Z_t^k f(X_t) \mathbbm{1}_{\{t < \tau_k\}}) | \\
&=|\mathbb{E}^{n,\beta}_x(Z_t f(X_t) \mathbbm{1}_{\{ t \geq \tau_k\}})| \\
&\leq C(f)~ |\mathbb{E}^{n,\beta}_x(Z_t \mathbbm{1}_{\{ t \geq \tau_k\}})| \\
&\leq C(f)~ \sup_{x \in D} |\mathbb{E}^{n,\beta}_x(Z_t \mathbbm{1}_{\{ t \geq \tau_k\}})|
\rightarrow 0 \quad \text{as } k \rightarrow \infty.
\end{align*}
uniformly on $D$ by (\ref{condZ}). Hence, $p^{n,\beta,\varrho}_t f$ is continuous on $D$ for each $d >0$ and so $p^{n,\beta,\varrho}_t f \in C_b(E)$.
\end{proof}

\begin{remark} \label{remZ}
Let $D \subset E$ be compact. Then $\lim_{k \rightarrow \infty} \sup_{x \in D} ~\mathbb{E}^{n,\beta}_x(\mathbbm{1}_{\{ \tau_k \leq t \}}~Z_t) =0$ holds for example if there exists some $p>1$ such that $\sup_{x \in D} \mathbb{E}^{n,\beta}_x(Z_t^p) < \infty$. Indeed, let $1<q<\infty$ such that $\frac{1}{p}+\frac{1}{q}=1$. Then 
\[ \sup_{x \in D} \mathbb{E}^{n,\beta}_x(\mathbbm{1}_{\{\tau_k \leq t\}} Z_t) \leq \sup_{x \in D} \mathbb{E}^{n,\beta}_x(Z_t^p)^{\frac{1}{p}}~ \sup_{x \in D} \big(\mathbb{P}^{n,\beta}_x(\tau_k \leq t)\big)^{\frac{1}{q}}.\]
Define $C_t:= \max_{i=1,\dots,n} \max_{0 \leq s \leq t} X_s^i$ for $t \geq 0$. Then for $x \in D$ and $k >d$
\[ \mathbb{P}^{n,\beta}_x( \tau_k \leq t) \leq \mathbb{P}^{n,\beta}_0( C_t \geq k-d ) \leq n~ \sqrt{\frac{t}{2 \pi}} \frac{4}{k-d} \exp(-\frac{(k-d)^2}{2t}) =: C(k) \rightarrow 0 \ \text{  as } k \rightarrow \infty \]
due to \cite[p.96,(8.3)']{KS98}, since the new time scale $(\tau(t))_{t \geq 0}$ fulfills $\tau(t) \leq t$ for every $t \geq 0$ and hence, $C_t \leq \max_{i=1,\dots,n} \max_{0 \leq s \leq t} |B_s^i|$ almost surely with respect to $\mathbb{P}^{n,\beta}_0$.
\end{remark}

\begin{proof}[Proof of Theorem \ref{thmmain1}]
By Section \ref{secgirsanov} there exists a strong Markov process $\mathbb{M}^{n,\beta,\varrho}$ with transition semigroup $(p^{n,\beta,\varrho}_t)_{t \geq 0}$ and the Dirichlet form associated to $\mathbb{M}^{n,\beta,\varrho}$ is given by the closure of $(\mathcal{E}^{n,\beta,\varrho},\mathcal{D})$ on $L^2(E;\varrho m_{n,\beta})$. Note that in this case $\mathcal{D} \cap D(L) \supset C_c^2(E)$ and $C_c^2(E)$ is also dense in $D(\mathcal{E}^{n,\beta})$. Indeed, Lemma \ref{lemdense} is based on the fact that $C_c^1([0,\infty))$ is dense for the one dimensional form which also holds for $C_c^2([0,\infty))$ (and even $C_c^{\infty}([0,\infty))$) by \cite[Theorem 5.2.8(i)]{ChFu11}. The strong Feller property is shown in Theorem \ref{strongfeller} and the last statement holds by \cite[Exercise 4.2.1]{FOT94}.
\end{proof}

\begin{proof}[Proof of Theorem \ref{thmmain2}]
The statement follows by the results proven in \cite[Corollary 4.18, Theorem 5.6]{FGV13} considering that the absolute continuity condition \cite[(4.2.9)]{FOT94} is fulfilled.
\end{proof}

\begin{remark} \label{remunique}
Assume additionally that
 \[ \mathbb{E}^{n,\beta}_x \big( \exp( \sum_{i=1}^n \int_0^t (\partial_i \ln \phi(X_s))^2~\mathbbm{1}_{(0,\infty)}(X_s^i)~ds) \big) < \infty\]
for every $x \in E$ and $t>0$. Then, the solution to (\ref{main}) is even unique in law in view of \cite[Chapter IV, Theorem 4.2]{WaIk89}.
\end{remark}

\section{Application to the dynamical wetting model} \label{appl}

\subsection{Densities corresponding to potential energies} \label{subsecPE}

In the following, let $\phi \in C^2(E)$ be strictly positive such that $\phi_{|E_+(B)} \in H^{1,2}(E_+(B))$ for every $\emptyset \neq B \subset I$. Set $H:=- \ln \phi$ (thus, $\phi=\exp(-H)$) and assume additionally that there exist real constants $K_1 \geq 0,~K_2$ and $K_3$ such that
\begin{enumerate}
\item $H(x) \geq -K_1$ for all $x \in E$,
\item $\partial_i H(x) \leq K_2$ for all $x \in \{ x_i=0\}:=\{ x \in E|~x_i=0\}$, $i=1,\dots,n$,
\item $\partial_i^2 H(x) \leq K_3$ for all $x \in E$, $i=1,\dots,n$.
\end{enumerate}
If we can verify (\ref{condZ}), Condition \ref{conditions} is fulfilled and thus, the results of Theorem \ref{thmmain1} and Theorem \ref{thmmain2} hold accordingly.\\

Using (\ref{AF}) and It{\^o}'s formula we see that
\begin{align}
M_t^{[\ln \phi]} - \frac{1}{2} \langle M^{[\ln \phi]} \rangle_t &= \sqrt{2} \sum_{i=1}^n \int_0^t \partial_i \ln \phi (X_s) \mathbbm{1}_{(0,\infty)} (X_s^i) dB_s^i -\sum_{i=1}^n \int_0^t \big(\partial_i \ln \phi(X_s) \big)^2 \mathbbm{1}_{(0,\infty)}(X_s^i) ds \notag \\
&=H(X_0)-H(X_t)+ \frac{1}{\beta} \sum_{i=1}^n \int_0^t \partial_i H(X_s) \mathbbm{1}_{\{0\}}(X_s^i) ds  \label{AFdecomp} \\
&\ \ + \sum_{i=1}^n \int_0^t \big( \partial_i^2 H(X_s) - \partial_i H(X_s)^2 \big) \mathbbm{1}_{(0,\infty)}(X_s^i) ds \notag \\
&\leq H(x) + K_1 + \frac{n}{\beta} K_2 t + n  K_3 t \notag
\end{align}
$\mathbb{P}^{n,\beta}_x$-a.s. for each $x \in E$.\\
Let $p>1$ be arbitrary, $D \subset E$ compact. Then it holds
\[ \sup_{x \in D}~ \mathbb{E}^{n,\beta}_x(Z_t^p) \leq \exp\big(p~( \sup_{x \in D} H(x) + K_1 + \frac{n}{\beta} K_2 t + n  K_3 t)\big) < \infty \quad \text{for every } t>0.\]
Thus, in view of Remark \ref{remZ}, (\ref{condZ}) holds true.

\subsection{Densities corresponding to potential energies given by pair potentials}

Assume that $H$ is given by a potential with nearest neighbor pair interaction, i.e., $H$ is defined as in (\ref{hamilt}). In particular, $\kappa:=\int_{\mathbb{R}} \exp(-V(r)) dr < \infty$, $V$ is convex, $V^{\prime}(0)=0$ and $V^{\prime}$ is non-decreasing. Then, we have $H(x) \geq -K_1$, $K_1 \geq 0$, since $V$ is bounded from below. Moreover,
\[ \partial_i H (x)= \frac{1}{2} \sum_{\stackunder{|i-j|=1}{j \in \{0,\dots,n+1\}}} V^{\prime}(x_i-x_j) \ \ \big(= \frac{1}{2} V^{\prime}(x_i-x_{i-1}) + \frac{1}{2} V^{\prime}(x_i-x_{i+1}) \big)
\]
for $i=1,\dots,n$ and
\[ \partial^2_i H (x)=\frac{1}{2} \sum_{\stackunder{|i-j|=1}{j \in \{0,\dots,n+1\}}} V^{\prime \prime}(x_i-x_j) \ \ \big(= \frac{1}{2} V^{\prime \prime}(x_i-x_{i-1}) + \frac{1}{2} V^{\prime \prime}(x_i-x_{i+1}) \big).
\]
Since $\partial_i H (x)=\frac{1}{2} (V^{\prime}(-x_{i-1}) + V^{\prime}(-x_{i+1})) \leq 0$ if $x_i=0$, we get $\partial_i H(x) \leq 0$ for all $x \in \{ x_i=0\}$ and furthermore, $\partial_i^2 H(x) \leq K_3$ , $i=1,\dots,n$, since $V^{\prime \prime}$ is bounded by a constant $c_{+}$. Thus, (i)-(iii) of Section \ref{subsecPE} are fulfilled. Furthermore, the condition $\phi_{|E_+(B)} \in H^{1,2}(E_+(B))$ for every $\emptyset \neq B \subset I$ with $\phi=\exp(-H)$ is satisfied in view of \cite[Remark 6.3, Remark 6.4]{FGV13}.

\subsection*{Acknowledgment}

We thank Torben Fattler for helpful comments and discussions. R.~Vo{\ss}hall gratefully acknowledges financial support in the form of a fellowship of the German state Rhineland-Palatine. Moreover, we thank an anonymous referee for helpful comments improving the readability of the paper.

\end{document}